\documentclass[a4paper,12pt]{amsart}
\usepackage{amsmath}
\usepackage{amsfonts}
\usepackage{amssymb}
\usepackage{amsthm}
\usepackage{multicol}
\newtheorem{thm}{Theorem}[section]

\newtheorem{lemma}[thm]{Lemma}
\newtheorem{prop}[thm]{Proposition}
\newtheorem{defn}[thm]{Definition}
\theoremstyle{definition}
\newtheorem{remark}[thm]{Remark}
\newtheorem{example}[thm]{Example}
\numberwithin{equation}{section}

\def\al{\alpha}
\def\be{\beta}
\def\ga{\gamma}
\def\de{\delta}

\def\la{\lambda}

\def\om{\omega}

\def\Ga{\Gamma}
\def\Z{\mathbb{Z}}

\def\R{\mathbb{R}}
\def\C{\mathbb{C}}
\def\N{\mathbb{N}}

\def\cH{\mathcal{H}}

\def\cK{\mathcal{K}}

\begin{document}
\title[The $J$-matrix method]{The $J$-matrix method}

\date{January 4, 2010}

\author{Mourad E.H. Ismail and Erik Koelink}
\address{Department of Mathematics, University of Central Florida, Orlando FL 32816, USA,
and King Saud University, Riyadh, Saudi
Arabia.}
\email{ismail@math.ucf.edu}
\address{Radboud Universiteit Nijmegen, IMAPP, FNWI, Heyendaalseweg 135, 6525 AJ Nijmegen, the Netherlands}
\email{e.koelink@math.ru.nl}

\dedicatory{dedicated to Dennis Stanton on his 60-th birthday}

\begin{abstract} Given an operator $L$ acting on a function space, the 
$J$-matrix method consists of finding a sequence $y_n$ of functions such that
the operator $L$ acts tridiagonally on $y_n$. Once such a tridiagonalization
is obtained, a number of characteristics of the operator $L$ can 
be obtained. In particular, information on eigenvalues and eigenfunctions, bound states,
spectral decompositions, etc. can be obtained in this way. We discuss
the general set-up and next two examples in detail; the 
Schr\"odinger operator with Morse potential and the Lam\'e equation. 
\end{abstract}

\keywords{tridiagonal operator, orthogonal polynomials, 
Schr\"{o}dinger operator with Morse potential, Lam\'e equation\\
Mathematics Subject Classification. 33C45, 42C05, 34L40}

\maketitle


\section{Introduction}\label{sec:introduction}

In many problems one is interested in the eigenfunctions
of an operator $L$ acting on some function space, or more
generally on the spectral decomposition of such an
operator when $L$ is self-adjoint. The purpose
of the paper is to give an introduction to a method that
has been successfully used on several occasions and at several places
in the literature but without a rigorous proof.  This method is known as the $J$-matrix
method or as tridiagonalization. A $J$-matrix, or a Jacobi operator,  
is a tridiagonal operator on some finite dimensional Hilbert space or on 
the sequence space $\ell^2(\N)$, which is usually assumed to
be symmetric and having no non-trivial closed reducing subspaces.
The last conditions are in general not imposed in this paper. 
A tridiagonalization of an operator $L$ acting on some function space is
given by a set of functions $\{ y_n\}_{n=0}^\infty$ 
such that $L$ acting on these functions is tridiagonal with respect to these
functions, i.e. such that \eqref{eq:gentridiagonalform} holds.
Note that in the particular case that the upper and lower diagonal
term vanishes, this just means that the functions $y_n$ are 
eigenfunctions for the operator $L$. 
Since there is an intimate relation between orthogonal polynomials
and three-term recurrence relations, see e.g. \cite{Chih}, \cite{Isma},
\cite{Koel}, \cite{Szeg}, \cite{Tesc}, in such a way that orthogonality properties of the
polynomials correspond
to the spectral properties of the corresponding Jacobi operator,
this can then be used to find information on eigenfunctions, spectral properties,
etc,  of the original operator $L$. 

This method is frequently used in physical and chemical models, see
e.g. \cite{AlHa-JPA2000}, \cite{AlHa-AnnPhys2004}, \cite{AlHa-PhysLett2004},
\cite{AlHa-AnnPhys2005}, \cite{AlHaBAAA}, \cite{BendD}, \cite{Broa-PhysRevA1978},
\cite{Broa-PhysRevA1982}, \cite{Broa-LNM}, \cite{Dies}, \cite{Hayd},
\cite{KnapD}, \cite{Mung}, \cite{SzmyZM}, \cite{YamaR-PhysRevA} and
references given there. It concerns mostly one-dimensional models, and 
the potentials and Hamiltonians discussed include
sextic, harmonic oscillator, (Dirac-) Coulomb, (Dirac-) Morse, etc.
Usually the papers mentioned start out with the operator $L$ to be
analyzed, and occasionally with the form of the polynomials prescribed,
e.g. as in  \cite{BendD} where the $y_n$ are monomials times a 
fixed function. 
This method is also closely related to the Lanczos algorithm in 
numerical analysis, see e.g. \cite[Ch.~2]{CullW}, and to related Krylov subspaces. 

The purpose of the paper is to discuss the method of tridiagonalization 
in a general fashion and to consider 
the case of the Lam\'e type operator, showing that
it can be tridiagonalized using Chebychev polynomials. This is motivated by the classical
theorem of Bochner \cite{Boch}, recalled in 
Theorem \ref{thm:Bochner}, which classifies all orthogonal
polynomials that are eigenfunctions to a second order differential
operator, see also \cite[Ch.~20]{Isma} for generalizations to difference
operators, and by the classification theorem of Al-Salam and
Chihara \cite{AlSaC}, recalled in Theorem \ref{thm:AlSalamChihara}, 
of orthogonal polynomials whose derivative can
be expressed in a simple way in terms of the orthogonal polynomials
themselves. In general it is difficult to say a priori if an operator can
be tridiagonalized, but in Section \ref{sec:generalsetup} we
prove this for a special class of operators including
second order differential and difference operators with
polynomial coefficients of some degree, and we discuss
an explicit example in Sections \ref{ssec:SchrodingerMorsepotential} and \ref{ssec:Lame}. 
If there is a way
to transform, e.g. by conjugation and/or change of variables,
to such a specific operator, then we can tridiagonalize the
resulting operator, as is the case for the examples 
in \S \ref{sec:secondorderdiffoper}. 

It should be noted that a Jacobi operator has simple spectrum, and
that conversely a self-adjoint operator with simple spectrum
can be realized as a Jacobi operator, see \cite[Ch.~VII]{Ston},
assuming that there are no non-trivial (closed) reducing 
subspaces, see also \cite{Bere}.  This is of particular interest in case of
the Schr\"odinger operator, where one can make use
of scattering theory in order to determine its 
spectral decomposition. 
In case both the tridiagonalization procedure works and 
the spectral decomposition can be made explicit by e.g. an
integral transformation, the methods can be linked
to each other leading to results for the special functions
and orthogonal polynomials involved and we discuss
an example for the Schr\"odinger equation with 
Morse potential due to Broad and Diestler, see 
\cite{Broa-LNM}, \cite{Dies}, 
\cite{Broa-PhysRevA1978}, \cite{Broa-PhysRevA1982},
\cite{KnapD}, as well as \cite{Koor-LNM}. 

The contents of the paper are as follows. In 
Section \ref{sec:generalsetup} we discuss a general
set-up for tridiagonalizable operators. In
Section \ref{sec:secondorderdiffoper} we restrict to
second order differential operators, where 
in particular we discuss the Broad-Diestler example
and the case of the Lam\'e operator. 

We want to point out that in many cases which 
are considered there is a link to the bispectral
problem, see e.g. \cite{Grun} for an introduction, and
that the tridiagonalization can be used for both the
operator in the geometric variable as for the operator
in the spectral variable. It is also to be pointed
out that one can also tridiagonalize (second order) difference operators,
which are included in the general scheme of 
Section \ref{sec:generalsetup}, and that one important 
example is already to be found in Groenevelt \cite{Groe} for the
case of the Wilson functions and the associated difference
operator. 
Finally, we want to mention two, closely related, possible extensions that can 
be useful as well. First, one can relate an operator to a doubly infinite
Jacobi matrix (i.e. acting on $\ell^2(\Z)$ instead of on $\ell^2(\N)$), 
see e.g. \cite{MassR}, \cite{Koel}, and \cite[Ch.~VII]{Bere}. 
As indicated by Berezanski{\u\i} \cite[Ch.~VII]{Bere} one can also consider this
case as $2\times 2$-matrix-valued variant of tridiagonalization, and this
can then be looked at from a matrix analogue of the tridiagonal 
situation, see e.g. \cite{DuraLR} for an introduction to 
matrix-valued orthogonal polynomials. This is useful for
such operators as the Dirac operator, see \cite{AlHa-AnnPhys2004},
\cite{AlHa-PhysLett2004}, \cite{Mung} and
also \cite{DuraG} in this context.


\section{The general set-up}\label{sec:generalsetup}

We consider first a special class of second order operators
that can be tridiagonalized, which is done in Section
\ref{ssec:motivationdef}. In Section \ref{ssec:symmetricTDoperators}
we moreover assume that this operator is symmetric, and we
consider possible self-adjoint extensions and their spectrum.

\subsection{Motivation and definition}\label{ssec:motivationdef}

Consider a linear operator $L$ acting on a suitable 
function space; typically $L$ is a differential
operator, or a difference operator. 
We look for linearly independent functions $\{ y_n\}_{n=0}^\infty$ 
such that $L$ is tridiagonal with respect to these
functions, i.e. there exist constants $A_n$, $B_n$, $C_n$ 
($n\in\N$)  such that 
\begin{equation}\label{eq:gentridiagonalform}
\begin{split}
L\, y_n &\, =\,  A_n\, y_{n+1} + B_n\, y_n + C_n\, y_{n-1}, \quad n\geq 1, \\
L\, y_0 &\, =\,  A_0\, y_1 + B_0\, y_0.
\end{split}
\end{equation}
We combine both equations by assuming $C_0=0$. 
It follows that $\sum_{n=0}^\infty p_n(z) \, y_n$ is 
a formal eigenfunction of $L$ for the eigenvalue $z$ if
$p_n$ satisfies
\begin{equation}\label{eq:recurrencepn}
z\, p_n(z) = C_{n+1}\, p_{n+1}(z) + B_n\, p_n(z) + A_{n-1}\, p_{n-1}(z)
\end{equation}
for $n\in \N$ with the convention $A_{-1}=0$. In case $C_n\not= 0$ for
$n\geq 1$, we can define $p_0(z)=1$ and use \eqref{eq:recurrencepn}
recursively to find $p_n(z)$ as a polynomials of degree $n$ 
in $z$. In case $A_n C_{n+1}>0$, $B_n\in\R$, $n\geq 0$, the polynomials $p_n$ are orthogonal
with respect to a positive measure on $\R$, and the measure
and its support then can give information on $L$ in case
$\{ y_n\}_{n=0}^\infty$ gives a basis for the function space
on which $L$ acts, or for $L$ restricted to the closure of the 
span $\{ y_n\}_{n=0}^\infty$ (which depends on the function
space under consideration).

We now consider a more specific form of the operator $L$.
Let $S$ be a linear operator acting on a suitable function space 
including the polynomials. We assume that $S$ preserves the space of 
polynomials, and that $S$ lowers the degree by $1$, i.e.
$S\, x^k = d_k x^{k-1}$, $k\in\N$, with $d_k\not= 0$ for
$k\geq 1$ and $d_0=0$. 
Similarly, $T$ is a linear operator acting on suitable function spaces 
including the polynomials. We assume that $T$ preserves the space of 
polynomials, and that $T$ lowers the degree by $2$, i.e.
$T\, x^k = d'_k x^{k-2}$, $k\in\N$, with $d'_k\not= 0$ for
$k\geq 2$ and $d'_0=d'_1=0$.

\begin{example} $T=S^2$, and $S= \frac{d}{dx}$, the $q$-derivative $S=D_q$,
or any other $q$-derivative, see e.g. \cite{Isma}. 
\end{example}

We now consider the operator $L$ on suitable function spaces 
\begin{equation}\label{eq:defgenoperator}
L\, =\, M_A\, T + M_B\, S + M_C
\end{equation}
where $M_f$ denotes the operator of multiplication
by a function $f$. We assume that $A$, $B$ and $C$ are
fixed polynomials with $\deg(A)=a$, $\deg(B)=b$ and
$\deg(C)=c$. In this case it follows that $L$ maps
a polynomial of degree $n$ in general to a 
polynomial of degree $\max(a+n-2, b+n-1,c+n)$. So 
if we look for a tridiagonalization in terms of
$y_n$ a polynomial of degree $n$ we require
$a\leq 3$, $b\leq 2$ and $c\leq 1$. 

The case $a\leq 2$, $b\leq 1$ has been studied extensively,
in particular the existence of polynomial eigenfunctions
for $M_A\, T + M_B\, S$ for $a\leq 2$, $b\leq 1$, see 
Bochner's Theorem \ref{thm:Bochner} for the classical
case of $T=S^2$, $S=\frac{d}{dx}$,
and for several other
instances of the operators $T$ and $S$, see 
\cite[Ch. 20]{Isma}. In most of these
cases $M_A\, T + M_B\, S$ have polynomial eigenfunctions
which are classes of orthogonal polynomials, so that 
$L \,= \, M_A\, T + M_B\, S + M_C$ is tridiagonal with respect to these polynomials
by the three-term recurrence relation in case $\deg(C)=1$.

So the previous discussion motivates why we
consider operators as in the following
definition.

\begin{defn}\label{def:TDoperator}
Let $S$ and $T$ be linear operators preserving
the space $\C[x]$ of polynomials, such that
$S$, respectively $T$, lowers the degree by $1$,
respectively $2$. We say that the linear operator
$L = M_A\, T + M_B\, S + M_C$ is a TD-operator if
$A$, $B$ and $C$ are polynomials with $\deg(A)=a\leq 3$, $\deg(B)=b\leq 2$ and
$\deg(C)=c\leq 1$ with $a=3$ or $b=2$. Here $M_f$
denotes multiplication by the function $f$. 
\end{defn}

\begin{thm}\label{thm:TDoperatorsaretridiagonalizable}
Let $L$ be a TD-operator, then there exist 
monic polynomials $\{ y_n\}_{n=0}^\infty$, $\text{\rm{deg}}(y_n)=n$,
such that \eqref{eq:gentridiagonalform} holds for
suitable coefficients $A_n$, $B_n$, $C_n$.
\end{thm}

\begin{proof} First note that there is no loss
by assuming the polynomials $y_n$ to be monic.

Recall we assume 
$S\, x^k = d_k x^{k-1}$, $k\in\N$, with $d_k\not= 0$ for
$k\geq 1$ and $d_0=0$, and $T\, x^k = d'_k x^{k-2}$, $k\in\N$, 
with $d'_k\not= 0$ for
$k\geq 2$ and $d'_0=d'_1=0$. Put
\begin{equation*}
\begin{split}
A(x) \, &= \, \al_3 \, x^3 +  \al_2 \, x^2 + \al_1 \, x + \al_0, \\ 
B(x) \, &= \, \be_2 \, x^2 + \be_1 \, x + \be_0, \\ 
C(x) \, &= \,  \ga_1 \, x + \ga_0. 
\end{split}
\end{equation*}
This implies
\begin{equation}\label{eq:thmTDoperatorsaretridiagonalizable1}
L\, x^k = (\al_3 d'_k + \be_2 d_k +\ga_1)\, x^{k+1} 
+ (\al_2 d'_k + \be_1 d_k +\ga_0)\, x^k + 
(\al_1 d'_k + \be_0 d_k)\, x^{k-1} 
+ \al_0 d'_k\, x^{k-2}. 
\end{equation}
In particular, the result follows with $y_n(x)=x^n$ in case $\al_0=0$.

Now take $y_0(x)=1$, so that $L\, y_0(x)= C(x)$. Putting
$y_1(x) = x+ c_0(1)$ we find that $L\, y_0 = A_0\, y_1 + B_0\, y_0$
if we take $A_0=\ga_1$, $\ga_1 c_0(1) + B_0=\ga_0$. Note that
there is a choice for the constant term $c_0(1)$ in $y_1$. 
Proceeding inductively, we assume that 
we have determined $\{y_0,\ldots, y_k\}$ such that
\begin{equation}\label{eq:thmTDoperatorsaretridiagonalizable2}
L\, y_n \, =\,  A_n\, y_{n+1} + B_n\, y_n + C_n\, y_{n-1}, 
\qquad 0 \leq n\leq k-1,
\end{equation}
Since $y_k$ and $y_{k+1}$ are monic polynomials, we see that 
\eqref{eq:thmTDoperatorsaretridiagonalizable2} to hold for
$n=k$ forces $A_k= \al_3 d'_k + \be_2 d_k +\ga_1$ by
\eqref{eq:thmTDoperatorsaretridiagonalizable1}. 
Putting $y_{k+1}(x) = x^{k+1}+\sum_{p=0}^k c_p x^p$, we see
that we need to determine $c_p$, $B_k$ and $C_k$ from
\begin{equation}\label{eq:thmTDoperatorsaretridiagonalizable3}
A_k c_p = \text{coeff}_p (Ly_k) - B_k \, \text{coeff}_p(y_k) 
- C_k\, \text{coeff}_p (y_{k-1}), \qquad 0\leq p \leq k,
\end{equation}
where $\text{coeff}_p (r)$ is the coefficient of $x^p$ in 
a polynomial $r$. Starting with $p=k$, we see that we need
to choose $c_k$, $B_k$ satisfying ---recall $y_k$ monic---
$A_k c_k = \text{coeff}_k (Ly_k) - B_k$, which can be easily
done for all values of $A_k$. So we fix $c_k$ and $B_k$. Next
for $p=k-1$ we get 
$A_k c_{k-1} = \text{coeff}_{k-1} (Ly_k) - B_k \, \text{coeff}_{k-1}(y_k) 
- C_k$, for which we choose a solution for $c_{k-1}$ and $C_k$. 
Now we have fixed $B_k$ and $C_k$, so we can solve $c_p$, $0\leq p \leq k-2$
uniquely (in case $A_k\not=0$) from 
\eqref{eq:thmTDoperatorsaretridiagonalizable3}, and we can assign some
value to $c_p$ in case $A_k=0$. 
\end{proof} 

\begin{remark}\label{rmk:thmTDoperatorsaretridiagonalizable}
(i) Note that in case e.g. $S=\frac{d}{dx}$ and $T=S^2$ one could
stop after the remark following \eqref{eq:thmTDoperatorsaretridiagonalizable1},
since we can use an affine transformation to reduce to the case
$A(0)=0$. However, in general we do not assume simple transformation
properties for $S$ and $T$. 

(ii) Note that there is freedom in the choice for $y_{n+1}$ in
the proof of Theorem \ref{thm:TDoperatorsaretridiagonalizable}. 
More requirements on the functions $y_n$ should indicate which
set to favour. 
\end{remark}

\subsection{Symmetric TD-operators}\label{ssec:symmetricTDoperators}

Now we assume that we have an Hilbert space $\cH$ of functions
containing the polynomials $\C[x]\hookrightarrow \cH$ injectively. 
We do not assume that 
$L$ can be extended as a bounded operator to $\cH$, but we 
assume that $L$ can be viewed as a densely defined operator
on $\cH$ such that $\C[x] \subset D(L)$, the domain of $L$.
Note that we assume $L\colon \C[x]\to \C[x]$, and we 
assume that $\C[x]$ dense in $\cH$ by switching to the
closure of $\C[x]$ in $\cH$ if necessary. 

\begin{prop}\label{prop:tridiagonalizationorthonormal} 
Let $L$ be a TD-operator. 
Assume $L$ with domain $D(L)$ is symmetric as
unbounded operator on $\cH$, then we can assume
$\langle y_n, y_m\rangle = 0$ for $n\not=m$. 
\end{prop}

\begin{proof} Since $\{ y_n\}_{n=0}^\infty$ is a family of
polynomials in $\cH$ we can apply the Gram-Schmidt procedure to
$\{y_0, y_1,\ldots\}$, and denote the resulting
orthogonal set of monic polynomials by $r_n$, then we
have $\deg(r_n)=\deg(y_n)=n$ and 
$r_n = y_n + \sum_{k<n} c_k y_k$. By 
\eqref{eq:gentridiagonalform} we find
$L\, r_n = A_n\, r_{n+1} + \sum_{k\leq n} c'_k r_k$.
Then $\langle L\, r_n, r_m\rangle = 0$ for $m>n+1$, 
and for $m<n-1$ we have
\begin{equation*}
\langle L\, r_n, r_m\rangle = 
\langle r_n, L^\ast r_m\rangle = 
\langle r_n, L\, r_m\rangle = 
\langle r_n, A_m\, r_{m+1} + \sum_{k\leq m} c'_k r_k \rangle = 0
\end{equation*}
since $m+1<n$. Note that $r_m \in \C[x] \subset D(L)\subset D(L^\ast)$. 
So $L$ is tridiagonal with respect to the orthogonal 
set $\{ r_n\}_{n=0}^\infty$. 
\end{proof}

Note that Proposition \ref{prop:tridiagonalizationorthonormal} easily
extends to $L$ skew-symmetric.

\begin{remark}\label{rmk:linktoeigenvalue} 
Assume that in Proposition \ref{prop:tridiagonalizationorthonormal}  
the orthogonal polynomials $y_n$ are 
eigenfunctions of a symmetric operator $D$, $D\, y_n=\la_n\, y_n$, such that 
$D$ preserves the polynomials, $D\colon \C[x]\to\C[x]$, and
the degree, $\deg(Dx^k) = k$, see e.g. Bochner's Theorem 
\ref{thm:Bochner} for the classical orthogonal polynomials and, 
more generally, for all polynomials in the Askey-scheme and its
$q$-analogue, see \cite{KoekS}. So in particular, we assume $\la_n\not=0$, $n\geq 1$. 
We assume that $D$ acts as a possibly unbounded 
linear operator on $\cH$. 
Let $X$ be the operator of 
multiplication by the independent variable, so that by orthogonality
\[
X\, y_n = a_n\, y_{n+1} + b_n\, y_n + c_n \, y_{n-1}.
\]
We also assume that $X\colon \C[x]\to \C[x]$ acts as a possibly unbounded self-adjoint
operator on $\cH$. 
Then the anticommutator $DX+XD$ is symmetric, and by
\begin{equation*}
(DX+XD)\, y_n = a_n(\la_{n+1}+\la_n)\, y_{n+1} + 
2\la_n b_n\, y_n + c_n (\la_n+\la_{n-1})\, y_{n-1}
\end{equation*}
it follows that $DX+XD$ is a symmetric TD-operator. 

Conversely, if $L$ is as in Proposition \ref{prop:tridiagonalizationorthonormal},
then we can define $D$ as a linear operator on $\C[x]$ by
\begin{equation}\label{eq:conversedefDfromL}
D\, x^n = \sum_{k=0}^{n-1} (-1)^k X^k\, L x^{n-1-k},\ n\geq 1, \qquad \quad
D\, 1 = 0, 
\end{equation}
by iterating $D\,x^n = DX \, x^{n-1} = (L-XD) x^{n-1}$ and using the
initial condition $D\, 1=0$.  Note that this completely determines
$D$ on the polynomials $\C[x]$. 
From \eqref{eq:conversedefDfromL} we
can show that $DX+XD=L$ on $\C[x]$. Since we assume $L$ and $X$ 
symmetric, we get $D^\ast X+XD^\ast =L$ on $\C[x]$ assuming 
$D^\ast$ preserves the polynomials. If one also assumes that
$\deg D^\ast x^k \leq k$, we see that $D^\ast$ must have the
same form as $D$.  So $D$ is symmetric if we can show that
$D^\ast 1= 0$. By the assumptions we have $D^\ast 1 = c$
for some constant $c$, and
\[
c = \frac{\langle D^\ast 1, 1\rangle}{ \| 1\|^2} = 
\frac{\langle 1, D\, 1\rangle}{ \| 1\|^2} = 0.
\]
Since $D$ is symmetric, preserving polynomials and the degree,
we find $D\, y_n = \la_n\, y_n$ for real $\la_n$ with
$\la_0=0$.

In case $L$ is antisymmetric, this has been completely worked out
by Koornwinder \cite[\S 2]{Koor-JCAM}, and then one has interesting
links to the so-called string equation. 
\end{remark}

In the situation of Proposition \ref{prop:tridiagonalizationorthonormal}
we can next orthonormalize the orthogonal polynomials $\{y_n\}_{n=0}^\infty$ in
$\cH$, and then we get
\begin{equation}\label{eq:symmetrictridiagonalform}
L\, y_n \, =\,  A_n\, y_{n+1} + B_n\, y_n + A_{n-1}\, y_{n-1}
\end{equation}
with $A_n, B_n\in\R$ and with the convention $A_{-1}=0$. 
Note that in the skew-symmetric case we obtain 
the same result but 
with $A_n, B_n\in i\R$ and with the convention $A_{-1}=0$.

The situation in \eqref{eq:symmetrictridiagonalform} is 
governed by the occurrences of $A_n=0$. In case 
$A_{n_1}=0$ and $A_{n_2}=0$ with $n_1<n_2$, and, in
view of the convention, $n_1=-1$
is allowed, we see that 
$L$ preserves the finite-dimensional subspace
spanned by $y_n$ for $n_1 < n \leq n_2$, which has dimension
$n_2-n_1$. In particular, if $n_2=n_1+1$ we see that 
$y_{n_2}$ is an eigenfunction of $L$ for the eigenvalue 
$B_{n_2}$. We have to distinguish between the 
cases of finite or infinite zeros of $n\mapsto A_n$.

\begin{thm}\label{thm:spectrumofLinfinitezerosofAn}
Let $(L, D(L))$, with $D(L)=\C[x]\hookrightarrow \cH$, be a symmetric 
densely defined TD-operator with the tridiagonalization 
\eqref{eq:symmetrictridiagonalform}. Assume
$-1=n_0 <n_1 <n_2 <\cdots$  is such that $A_{n_i}=0$, 

\noindent
\textrm{\rm (i)} In case $\N\ni n\mapsto A_n$ has an infinite
number of zeros, the finite-dimensional subspaces
$\cH_i = \text{\rm span}\{ y_n \mid n_{i-1} < n \leq n_i\}$,
$i\geq 1$, $\dim \cH_i = n_i-n_{i-1}$, are invariant for $L$. 
Moreover, $\cH = \bigoplus_{i=1}^\infty \cH_i$ and 
$L \vert_{\cH_i}$ has simple spectrum consisting of
$\dim \cH_i$ different eigenvalues. The operator
$(L, D(L))$ is essentially self-adjoint. 

\noindent
\textrm{\rm (ii)} In case $\N\ni n\mapsto A_n$ has $k$ zeros,
$n_0=-1 < n_1 < \cdots  < n_k$, the $k$ finite dimensional 
subspaces 
$\cH_i = \text{\rm span}\{ y_n \mid n_{i-1} < n \leq n_i\}$,
$1\leq i \leq k$, $\dim \cH_i = n_i-n_{i-1}$, are invariant for $L$. 
$L \vert_{\cH_i}$ has simple spectrum consisting of
$\dim \cH_i$ different eigenvalues. Consider the sequence
of polynomials determined by $p_0(z)=1$, and 
\[
z\, p_n(z) = A_{n+n_k+1}\, p_{n+1}(z) + B_{n+n_k+1}\, p_n(z) 
+ A_{n+n_k} p_{n-1}(z) 
\]
then $(L,D(L))$ is essentially self-adjoint if and only if
the orthogonal polynomials $\{ p_n\}_{n=0}^\infty$ correspond
to a determinate moment problem. 
\end{thm}

\begin{proof}  In case $A_{n_1}=0$ and $A_{n_2}=0$ with $n_1<n_2$ we see that 
$L$ preserves the finite-dimensional subspace $\cK$, $\dim \cK= n_2-n_1$,
spanned by $y_n$ for $n_1 < n \leq n_2$. By \eqref{eq:symmetrictridiagonalform}
it follows that $L\colon \cK \to \cK$ is given by a Jacobi matrix, i.e.
a symmetric tridiagonal matrix. It is well-known, see e.g. 
\cite{Ston}, \cite{Szeg}, that such a matrix has
$\dim \cK$ different eigenvalues, and that each of them has multiplicity one.
In case (i) we have that the closure of $L$ is given by 
its maximal extension, which is self-adjoint.

In case (ii) the previous considerations remain valid for the finite
dimensional invariant subspaces, and we are left with the study of
the action of $L$ on the closure $\cK$ of the linear span $\{ y_{n+n_k}\}_{n=0}^\infty$.
Let $\ell^2(\N)$ be the Hilbert space of square summable sequences 
with standard orthonormal basis $\{e_n\}_{n=0}^\infty$. Then 
$U\colon \cK\to \ell^2(\N)$, $y_{n+n_k} \mapsto e_n$ is a unitary map such that
\[
U \, L\, U^\ast \, e_n = A_{n+n_k+1}\, e_{n+1} + B_{n+n_k+1}\, e_n 
+ A_{n+n_k} e_{n-1}.
\]
So the action of $L$ restricted to $\cK$ is intertwined with the 
action of a Jacobi operator on $\ell^2(\N)$, and it is well-known,
see e.g.  \cite{Koel}, \cite{Ston}, \cite{Tesc}, that
this Jacobi operator is essentially self-adjoint if and only if
the corresponding moment problem is determinate.
\end{proof}

The spectrum of a TD-operator on finite-dimensional invariant
subspaces can be determined explicitly, and in case 
we can also find the eigenfunctions in another (direct) way
this leads to non-trivial sums, see e.g. \S \ref{ssec:SchrodingerMorsepotential}
for an example. 
Let us now assume that the TD-operator $L$ with domain $D(L)$
acting on $\cH$ with $D(L)=\C[x]\hookrightarrow\cH$ dense in $\cH=L^2(\nu)$ 
is essentially self-adjoint, and we assume that $A_n$ has no
zeros, except the convention $A_{-1}=0$. So we are in the
second case of Theorem \ref{thm:spectrumofLinfinitezerosofAn}.  
In this case the spectrum is
simple \cite[Ch.~VI]{Ston}, so the spectral theorem states that 
there exists a unitary map $\Upsilon \colon \cH=L^2(\nu) \to 
\cK = L^2(\mu)$, to some weighted $L^2$-space with $\mu$ a positive
Borel measure on $\R$ such that
$\Upsilon\, L\, \Upsilon^\ast = X$, where $X$ is the (possibly)
unbounded operator on $L^2(\mu)$ of multiplication by the
independent variable, say $\la$, see  \cite[Ch.~VI]{Ston}. We assume that there exist 
suitable functions $\phi_\la$, generally not assumed to be
in the Hilbert space $\cH$, such that 
$\bigl(\Upsilon f\bigr) (\la) = \langle f, \phi_\la\rangle$ for
suitable $f\in\cH$ and where $L\, \phi_\la = \la\, \phi_\la$. 
This is a typical situation in the spectral decomposition of
various second order differential or difference operators. 

In this case $L$ has simple spectrum, and since $\Upsilon y_n$
satisfies the same recurrence relation we find
\begin{equation}\label{eq:integraltransform}
\bigl( \Upsilon y_n\bigr)(\la) = 
\int \phi_\la(x)\, y_n(x)\, d\nu(x) \, = \, p_n(\la)\, (\Upsilon 1)(\la),
\end{equation}
or, the integral transform with kernel the (formal) eigenfunctions of 
$L$ maps the orthogonal polynomials $y_n$ to the orthogonal polynomials
$p_n$, up to a common multiple.  See e.g. \cite{Groe}, \cite{Koor-LNM}
for examples. 


\section{Second order differential operators}\label{sec:secondorderdiffoper}

We now restrict ourselves to the case of second order differential operators
as an example. Needless to say that appropriate
$q$-analogues or difference analogues can be considered as well
within  this general framework. First we discuss some
generalities, and then we discuss two examples; the 
Schr\"odinger equation with the Morse potential 
in Section \ref{ssec:SchrodingerMorsepotential}, and
the Lam\'e equation in Section \ref{ssec:Lame}. 

\subsection{Theorems by Bochner and Al-Salam--Chihara}\label{ssec:thmBochnerAlSalamChihara}
We now assume that 
\begin{equation}\label{eq:Las2ndorderODE}
L = M_A \, \frac{d^2}{dx^2} + M_B\, \frac{d}{dx} 
+ M_C, 
\end{equation}
so we take $S=\frac{d}{dx}$, and $T=S^2= \frac{d^2}{dx^2}$.
This then fits into the scheme of 
Theorem \ref{thm:TDoperatorsaretridiagonalizable}. Recall
our basic assumption that $\deg(A)=a\leq 3$, 
$\deg(B)=b\leq 2$, and $\deg(C)=c\leq 1$, and that we assume that
$a = 3$ or  $b = 2$. Indeed, in case $a\leq 2 $ and $b \leq 1$ we
are essentially back to Bochner's Theorem \ref{thm:Bochner}, and the
fact that all polynomials in Bochner's Theorem satisfy a three-term
recurrence. For completeness, we recall Bochner's Theorem here,
see Bochner \cite{Boch}, or e.g. \cite[\S 20.1]{Isma}.

\begin{thm}[Bochner (1929)]\label{thm:Bochner}
Up to affine scaling the only sets $\{ y_n\}_{n=0}^\infty$ of 
polynomials that are eigenfunctions to a second order differential operator
$A(x)\, y_n''(x) \, + \, B(x)\, y_n'(x)  \, +\, \la_n\, y_n(x)=0$ for all $n\geq 0$ are
\begin{enumerate}
\item Jacobi polynomials: $\deg(A)=2$ with different zeroes, $\deg(B)=0$ or $1$;
\item Laguerre polynomials: $\deg(A)=1$, $\deg(B)=1$; 
\item Hermite polynomials: $\deg(A)=0$, $\deg(B)=1$;
\item Bessel polynomials: $\deg(A)=2$ with double zero, $\deg(B)=0$ or $1$ and
$A$ and $B$ have no common zero;
\item Monomials: $\deg(A)=2$ with double zero, $\deg(B)=1$ and
$A$ and $B$ have a common zero.
\end{enumerate}
\end{thm}

For the notation of orthogonal polynomials we follow the notation as
in \cite{KoekS}, and for the Bessel polynomials we follow 
\cite{Isma}, so $P^{(\al,\be)}_n(x)$, $L^{(\al)}_n(x)$ and $H_n(x)$ denote
Jacobi, Laguerre and Hermite polynomials, whereas $y_n(x;a,b)$
are Bessel polynomials. 
In Bochner's Theorem
\ref{thm:Bochner} the first three sets are orthogonal polynomials 
on the real line and these  are classical orthogonal polynomials.
The Bessel polynomials are not orthogonal on the real line with respect to
a positive measure, see \cite[\S 4.10]{Isma}, and the same
is true for the monomials.

\begin{remark}  
Bochner's theorem has several analogues, e.g. by replacing
the differential operator $\frac{d}{dx}$ by one of the $q$-difference 
operators, see \cite[Ch.~20]{Isma} for more information. Bochner's work is
predated by Routh's paper \cite{Rout} in which Routh looks for 
polynomial solutions to a second order differential operator with 
polynomial coefficients of degree at most two and one which also
satisfy a three-term recurrence relation, see also \cite[Ch.~20]{Isma}.
\end{remark}

Next one can ask for a relation between the derivative of an
orthogonal polynomial, and its relation to orthogonal polynomials
of possibly different degree within the same family. 
The following classification theorem has been obtained by
Al-Salam and Chihara \cite{AlSaC}, see also the survey
by Al-Salam \cite{AlSaOP1990}.

\begin{thm}[Al-Salam and Chihara (1972)]\label{thm:AlSalamChihara}
If $\{ p_n\}_{n=0}^\infty$ is a family of orthogonal polynomials
on the real line 
with differential-recursion relation
\begin{equation*}
G(x)\, \frac{dp_n}{dx}(x) = {\mathcal A}_n\, p_{n+1}(x) + 
{\mathcal B}_n\, p_n(x)  + {\mathcal C}_n\, p_{n-1}(x)
\end{equation*}
for constants ${\mathcal A}_n$, ${\mathcal B}_n$, ${\mathcal C}_n$
and $G$ (necessarily) a polynomial of degree $\leq 2$, then 
the $p_n$'s are (up to affine scaling) Jacobi, Laguerre
or Hermite polynomials: 
\begin{equation*}
\begin{split}
(1-x^2) \, \frac{d}{dx}P^{(\al,\be)}_n(x) &\, = 
{\mathcal A}^{(\al,\be)}_n\, \, P^{(\al,\be)}_{n+1}(x) +
{\mathcal B}^{(\al,\be)}_n\, \, P^{(\al,\be)}_n(x) +
{\mathcal C}^{(\al,\be)}_n\, \, P^{(\al,\be)}_{n-1}(x); \\
x\frac{d}{dx} L^{(\al)}_n(x) &\, = \, n\, L^{(\al)}_n(x) \, 
- (n+\al)\, L^{(\al)}_{n-1}(x); \\ 
\frac{d}{dx} H_n(x) &\, = \, 2n\, H_{n-1}(x). 
\end{split}
\end{equation*}
\end{thm}

\begin{remark}\label{rmk:thmAlSalamChihara}
The Al-Salam--Chihara classification Theorem \ref{thm:AlSalamChihara}  
concerns orthogonal
polynomials, so that the Bessel polynomials and the monomials
do not occur in the list. However, the Bessel polynomials and the
monomials satisfy a differential-recursion relation of the
form
\begin{equation*}
\begin{split}
x^2 \frac{d}{dx} y_n(x;a,b) &\, = {\mathcal A}^{a,b}_n\, y_{n+1}(x;a,b) +
{\mathcal B}^{a,b}_n\, y_n(x;a,b) +
{\mathcal C}^{a,b}_n\, y_{n-1}(x;a,b), \qquad
\frac{d}{dx} x^n \,= \, n x^{n-1}.
\end{split}
\end{equation*}
So Bochner's Theorem \ref{thm:Bochner} and the 
Al-Salam--Chihara Theorem \ref{thm:AlSalamChihara}
deal with the same sets of polynomials. 
\end{remark}

As noted in Remark \ref{rmk:thmTDoperatorsaretridiagonalizable}(i), 
by an affine transformation 
we can assume $A(0)=0$ in \eqref{eq:Las2ndorderODE}, and
then $L$ is tridiagonalized by the monomials, cf. 
Theorem \ref{thm:TDoperatorsaretridiagonalizable} and its proof.  
However, there is choice in the polynomials leading
to tridiagonalization. 
Using Bochner's Theorem \ref{thm:Bochner}, 
the Al-Salam--Chihara Theorem \ref{thm:AlSalamChihara}
and Remark \ref{rmk:thmAlSalamChihara} one can proceed as
follows to tridiagonalize the operator $L$: first 
use Bochner's Theorem \ref{thm:Bochner} to get rid
of the second order derivative; secondly use
Theorem \ref{thm:AlSalamChihara} to get rid of the 
the first order derivative. Note that there many choices available.
Firstly by using an affine
transformation; secondly by choosing the parameters in case of
the Jacobi, Laguerre, or Bessel polynomials; and thirdly
in the possible decomposition of the polynomial
$A$ as a product of two lower order polynomials. 
We give an example of this procedure when discussing
the Lam\'e equation. 

\subsection{Symmetric second order differential equations}\label{ssec:symmetricsecordde}

We now consider the case of 
$L = M_A \, \frac{d^2}{dx^2} + M_B\, \frac{d}{dx} 
+ M_C$ being symmetric on a Hilbert space $\cH= L^2\bigl( (a,b), w(x)dx\bigr)$,
where $-\infty \leq a<b\leq \infty$ and $w(x)>0$ on $(a,b)$. Recall
that we assume $\C[x] \hookrightarrow \cH$ as a dense subspace.

\begin{lemma}\label{lem:diffeqLsymmetric}
Assume $A$, $B$ and $C$ are real-valued polynomials on $\R$. 
Moreover, assume $w\in C^1(a,b)$, and $(Aw)'=Bw$, then
$L$, with domain $D(L)\, =\, C^\infty_c(a,b)$, is a symmetric 
operator on $\cH = L^2\bigl( (a,b), w(x)dx\bigr)$.
\end{lemma}


\begin{proof} For $f,g \in C^\infty_c(a,b)$ we have
\begin{equation}\label{eq:lemdiffeqLsymmetric1}
\begin{split}
&\, \int_a^b \bigl( Lf\bigr) (x) \overline{g(x)}\, w(x)\, dx  
= \int_a^b \bigl( C(x)f(x) + B(x)f'(x) + A(x)f''(x)\bigr) \overline{g(x)}\, w(x)\, dx
\\ =\, &\,  \int_a^b  C(x)f(x)\overline{g(x)}w(x)\, dx  + 
\int_a^b  f'(x) B(x)\overline{g(x)} w(x)\, dx \\  & \qquad\qquad -  
\int_a^b A(x) f'(x)\overline{g'(x)}w(x)\, dx 
- \int_a^b  f'(x)\overline{g(x)} \bigl( Aw\bigr)'(x)\, dx
\\ =\, &\,  \int_a^b  C(x)f(x)\overline{g(x)}w(x)\, dx  
 - \int_a^b A(x) f'(x)\overline{g'(x)}w(x)\, dx
\end{split}
\end{equation}
since $(Aw)'=Bw$. The right hand side of 
\eqref{eq:lemdiffeqLsymmetric1} yields the symmetry. 
\end{proof}

Assuming the conditions of Lemma \ref{lem:diffeqLsymmetric} 
on the weight function $w$, we can write, for $f,g \in \C[x]$,
\begin{equation*}
\begin{split}
\langle L\, f, g\rangle =&\,  
\int_a^b C(x) f(x)\overline{g(x)}\, w(x)\, dx - 
\int_a^b A(x) f'(x) \overline{g'(x)}\, w(x)\, dx  \\
&\, + A(b)w(b)\, f'(b)\overline{g(b)} - A(a)w(a)\, f'(a)\overline{g(a)},
\end{split}
\end{equation*}
so that Lemma \ref{lem:diffeqLsymmetric} has the following
analogue in case the domain $D(L)=\C[x]$ is considered. 

\begin{lemma}\label{lem:diffeqLsymmetricpols}
Assume the conditions of Lemma \ref{lem:diffeqLsymmetric} on $A$, $B$, $C$
and $w$ hold. Moreover, assume
$Aw$ has a zero in $a$ and $b$, which has to 
be interpreted for $a=\infty$, respectively $b=-\infty$, as
$\lim_{x\to \infty} w(x)p(x)=0$, respectively $\lim_{x\to -\infty} w(x)p(x)=0$,
for all polynomials $p$.  
Then $L$, with domain $D(L)=\C[x]$, is a symmetric 
operator on $\cH = L^2\bigl( (a,b), w(x)dx\bigr)$.
\end{lemma}

The first order differential equation for the weight function 
$w$ in Lemma \ref{lem:diffeqLsymmetric} is rewritten as
\begin{equation}\label{eq:firstorderodeweight}
(\ln w)'= \frac{w'}{w} = \frac{B-A'}{A},
\end{equation}
i.e. the logarithmic derivative of $w$ is a rational 
function for which the degree of the numerator polynomial
is at most $2$ and the degree of the denominator
polynomial is at most $3$. Depending on the structure of
the rational function the differential equation 
\eqref{eq:firstorderodeweight} can be solved 
straightforwardly using a partial fraction decomposition. 
The solution of \eqref{eq:firstorderodeweight}
will very much depend on the relation between the polynomials
$A$ and $B$.

E.g. in the special case $A'=B$ we see
that we can take $w(x)=1$, and Lemma
\ref{lem:diffeqLsymmetric} applies, and from  
Lemma \ref{lem:diffeqLsymmetricpols} we see that $L$ with $D(L)=\C[x]$ 
is symmetric on $L^2\bigl( (a,b), dx\bigr)$
if $a$ and $b$ are different zeroes of the polynomial $A$, $a<b$.

\subsection{Schr\"odinger equation with Morse potential}\label{ssec:SchrodingerMorsepotential}

The Schr\"odinger equation with Morse potential is studied by
Broad \cite{Broa-LNM} and Diestler \cite{Dies} in the study of 
a larger system of coupled equations used in
modelling atomic dissocation. The Schr\"odinger equation with Morse potential
is used to model a two-atom molecule in this larger system. 

The Schr\"odinger equation with Morse potential is
\begin{equation}\label{eq:SchrodingerMorsepotential}
-\frac{d^2}{dx^2} + q, \qquad q(x) = b^2(e^{-2x}-2e^{-x}), 
\end{equation}
which is an unbounded operator on $L^2(\R)$. Here $b>0$
is a constant. 
It is a self-adjoint operator with respect to its form domain,
see \cite[Ch.~5]{Sche} and $\lim_{x\to\infty} q(x) =0$, and
$\lim_{x\to-\infty} q(x) =+\infty$. 
Note $\min (q)=-b^2$, so that the discrete spectrum is contained
in $[-b^2,0]$ and it consists of isolated points.
We look for solutions to $-f''(x) + q(x)f(x)=\ga^2 f(x)$.
Put $z= 2be^{-x}$ so that $x\in\R$ corresponds to $z\in(0,\infty)$,
and let $f(x)$ correspond to $\frac{1}{\sqrt{z}} g(z)$, then
\begin{equation}\label{eq:Whittakereq}
g''(z) + \frac{(-\frac14 z^2+bz +\ga^2 +\frac14)}{z^2} g(z) = 0.
\end{equation}
which is precisely the Whittaker equation 
with $\kappa=b$, $\mu=\pm i\ga$, and
the Whittaker integral transform gives the spectral decomposition
for this Schr\"{o}dinger equation, see \cite[\S~IV]{Fara}.
In particular, depending on the value of $b$ the
Schr\"odinger equation has finite discrete spectrum, i.e.
bound states, see the Plancherel formula \cite[\S~IV]{Fara}, and
in this case the Whittaker function terminates and can be written
as a Laguerre polynomial of type $L_m^{(2b-2m-1)}(x)$, 
for those $m\in \N$ such that $2b-2m>0$.

The Schr\"{o}dinger operator is transformed into a TD-operator, and
a particularly nice basis in which the operator is tridiagonal
is obtained by Broad \cite{Broa-LNM} and Diestler \cite{Dies}.
Put $N=\# \{ n\in\N \, |\,  n < b-\frac12 \}$, i.e. 
$N= \lfloor b+\frac12 \rfloor$, so that $2b-2N > -1$, and we assume
for simplicity $b\notin \frac12+\N$. 
Let $T\colon L^2(\R)\to L^2((0,\infty); z^{2b-2N}e^{-z}dz)$ be the map
$(Tf)(z) = z^{N-b-\frac12}e^{\frac12 z}\, f(\ln(2b/z))$, then
$T$ is unitary, and
$T\bigl(-\frac{d^2}{dx^2} + q\bigr) T^\ast=L$ with
$L= M_A \frac{d^2}{dz^2} + M_B \frac{d}{dz} + M_C$ with
$A(z)=-z^2$, $B(z)=(2N-2b-2 + z)z$, $C(z)=-(N-b-\frac12)^2 + z(1-N)$. 
Using the second-order differential equation, see 
e.g. \cite[(4.6.15)]{Isma}, \cite[(1.11.5)]{KoekS}, \cite[(5.1.2)]{Szeg}, for
the Laguerre polynomials, cf. Bochner's Theorem \ref{thm:Bochner}, the three-term recurrence
relation for the Laguerre polynomials, 
see e.g. \cite[(4.6.26)]{Isma}, \cite[(1.11.3)]{KoekS}, \cite[(5.1.10)]{Szeg},
and the differential-recursion formula as in Theorem 
\ref{thm:AlSalamChihara} for the Laguerre polynomials we find that this
operator is tridiagonalized by the Laguerre polynomials
$L_n^{(2b-2N)}$. When we translate this back to the Schr\"{o}dinger
operator we started with we obtain 
\begin{equation}\label{eq:defyn}
y_n(x) = (2b)^{(b-N+\frac12)}\sqrt{\frac{n!}{\Ga(2b-2N+n+1)}}
e^{-(b-N+\frac12)x} e^{-be^{-x}}
\, L^{(2b-2N)}_n(2be^{-x}) 
\end{equation}
as an orthonormal basis for $L^2(\R)$ such that 
\begin{equation}\label{eq:ttrpart13}
\begin{split}
\Bigl(-\frac{d^2}{dx^2} + q\Bigr) y_n \, = &
\, -(1-N+n)\sqrt{(n+1)(2b-2N+n+1)}\, y_{n+1}  \\ 
\, &+\, \Bigl( -(N-b-\frac12)^2 + (1-N+n)(2n+2b-2N+1) - n  \Bigr)\, y_n \\
\, &\, -(n-N) \sqrt{n(2b-2N+n)}\,\,  y_{n-1}.
\end{split}
\end{equation}
Note that \eqref{eq:ttrpart13} is written in a symmetric 
tridiagonal form. 

The space $\cH^+$ spanned by $\{ y_n\}_{n=N}^\infty$ and the
space $\cH^-$ spanned by $\{ y_n\}_{n=0}^{N-1}$ are invariant with
respect to $-\frac{d^2}{dx^2} + q$ which follows 
from \eqref{eq:ttrpart13}.
Note that $L^2(\R) = \cH^+ \oplus \cH^-$, $\dim (\cH^-)=N$.

In order to determine the spectral properties of the
Schr\"{o}dinger operator in this way we follow the approach
of Theorem \ref{thm:spectrumofLinfinitezerosofAn}.
We first consider
its restriction on the finite-dimensional invariant subspace
$\cH^-$. We look for eigenfunctions $\sum_{n=0}^{N-1} P_n(z)\, y_n$ for
eigenvalue $z$, so we need to solve 
\begin{equation*}
\begin{split}
z\, P_n(z) &\, = (N-1-n)\sqrt{(n+1)(2b-2N+n+1)}\, P_{n+1}(z)  
\\ &\qquad + \, \Bigl( -(N-b-\frac12)^2 + (1-N+n)(2n+2b-2N+1) - n  \Bigr)\, P_n(z)  
\\ &\qquad +  (N-n) \sqrt{n(2b-2N+n)}\,  P_{n-1}(z), \qquad 0\leq n\leq N-1.
\end{split}
\end{equation*}
which corresponds to some orthogonal polynomials on a finite 
discrete set. These polynomials are expressible in terms
of the dual Hahn polynomials, see \cite[\S 6.2]{Isma}, \cite[\S 1.6]{KoekS}, and we find that
$z$ is of the form $-(b-m-\frac12)^2$, $m$ a nonnegative integer
less than $b-\frac12$, and 
\begin{equation*}
\begin{split}
&\, P_n(-(b-m-\frac12)^2) = \sqrt{\frac{(2b-2N+1)_n}{n!}}
\, R_n(\la(N-1-m); 2b-2N,0, N-1), 
\end{split}
\end{equation*}
using the notation of \cite[\S 6.2]{Isma}, \cite[\S 1.6]{KoekS}. 
Since we have now two expressions for the
eigenfunctions of the Schr\"{o}dinger operator for a 
specific simple eigenvalue, we obtain, 
after simplifications,
\begin{equation}\label{eq:expansionLaguerreanddualHahn}
\begin{split}
&\, \sum_{n=0}^{N-1} \, R_n(\la(N-1-m); 2b-2N,0, N-1)\,
L^{(2b-2N)}_n(z)\,   =\, C\,   z^{N-1-m} \, L^{(2b-2m-1)}_m(z), \\ 
&\, C\, = \, (-1)^{N+m+1} \left( (N+m-2b)_{N-1-m} \binom{N-1}{m}\right)^{-1}
\end{split}
\end{equation}
where the constant $C$ can be determined by e.g. considering leading
coefficients on both sides. Using the orthogonality relations
\cite[(1.6.2)]{KoekS} of the dual Hahn polynomials, 
\eqref{eq:expansionLaguerreanddualHahn} can be inverted. 

On the invariant subspace $\cH^+$ we look for formal eigenvectors 
$\sum_{n=0}^\infty P_n(z)\ y_{N+n}(x)$ for the eigenvalue $z$. 
This leads to the recurrence relation 
\begin{equation*}
\begin{split}
z\, P_n(z) &\, =  -(1+n)\sqrt{(N+n+1)(2b-N+n+1)}\,  P_{n+1}(z) \\
&\qquad + \bigl( -(N-b-\frac12)^2 +(1+n)(2n+2b+1)- n-N\bigr)\, P_n(z) \\
&\qquad
-n \sqrt{(N+n)(2b-N+n)}\, P_{n-1}(z).
\end{split}
\end{equation*}
This corresponds with the three-term recurrence relation for the
continuous dual Hahn polynomials, see \cite[\S 1.3]{KoekS}, 
with $(a,b,c)$ replaced by  $(b+\frac12, N-b+\frac12, b-N+\frac12)$, and
note that the coefficients $a$, $b$ and $c$ are strictly positive. 
We find, with $z=\ga^2\geq 0$ 
\begin{equation*}
\begin{split}
P_n(z) &\, = 
\frac{S_n(\ga^2;b+\frac12, N-b+\frac12, b-N+\frac12)}
{n! \sqrt{(N+1)_n\, (2b-N+1)_n}} 
\end{split}
\end{equation*}
and these polynomials satisfy 
\begin{equation*}
\begin{split}
&\, \int_0^\infty P_n(\ga^2) P_m(\ga^2) \, w(\ga)\, d\ga = \de_{n,m}, 
\\ 
&\, w(\ga) =\frac{1}{2\pi\, N!\, \Ga(2b-N+1)} \left| \frac{\Ga(b+\frac12 +i\ga)\Ga(N-b+\frac12 +i\ga)\Ga(b-N+\frac12 +i\ga)}{\Ga(2i\ga)}\right|^2.
\end{split}
\end{equation*}
Note that the series $\sum_{n=0}^\infty P_n(\ga^2)\, y_{N+n}$ diverges in 
$\cH^+$ (as a closed subspace of $L^2(\R)$). Using 
the results on spectral decomposition of Jacobi operators, we 
obtain the spectral decomposition
of the Schr\"{o}dinger operator restricted
to $\cH^+$ as 
\begin{equation}
\begin{split}
&\, \Upsilon\colon  \cH^+ \to L^2((0,\infty); w(\ga)\, d\ga), 
\qquad \bigl( \Upsilon y_{N+n}\bigr)(\ga) = P_n(\ga^2), \\
&\, \langle ( -\frac{d^2}{dx^2} + q ) f, g\rangle = 
\int_0^\infty \ga^2 (\Upsilon f)(\ga)  \overline{(\Upsilon g)(\ga)}\, w(\ga)\, d\ga
\end{split}
\end{equation}
for $f,g\in \cH^+\subset L^2(\R)$ such that $f$ is in the domain
of the Schr\"odinger operator. 

In this way we have obtained the spectral decomposition of the Schr\"odinger
operator on the invariant subspaces $\cH^-$ and $\cH^+$, where 
the space $\cH^-$ is spanned by the bound states, i.e. by the eigenfunctions
for the negative eigenvalues, and $\cH^+$ is the reducing subspace on
which the  Schr\"odinger
operator has spectrum $[0,\infty)$. 
The link between the two approaches for the discrete spectrum is
given by \eqref{eq:expansionLaguerreanddualHahn}. For the
continuous spectrum it leads to the fact that the Whittaker
integral transform maps Laguerre polynomials to continuous
dual Hahn polynomials, and we can interpret
\eqref{eq:expansionLaguerreanddualHahn} also in this way.
For explicit formulas we refer to 
\cite[(5.14)]{Koor-LNM}. Koornwinder \cite{Koor-LNM} generalizes
this to the case of the Jacobi function transform mapping 
Jacobi polynomials to Wilson polynomials, which in turn has been 
generalized by Groenevelt \cite{Groe} to the Wilson function
transform mapping Wilson polynomials to Wilson 
polynomials.

\subsection{Lam\'e equation}\label{ssec:Lame}
The classical Lam\'e equation is 
$\frac{d^2 F}{du^2}(u) - \bigl( m(m+1) \wp(u) + E \bigr)F(u)=0$. 
Here $\wp$ is the Weierstra\ss\ $\wp$-function, which is 
a doubly-periodic function with periods
$2\om_1$, $2\om_2$ (and $\frac{\om_1}{\om_2}\not\in\R$).
We do not yet assume a condition on $m\in \R$, but note
the symmetry $m\leftrightarrow-m-1$. This equation is
very classical, and it is studied in \cite[\S 23]{WhitW}
in detail. 
Put 
$x=\wp(u)$, and $F(u)=f(\wp(u))$ then 
\begin{equation}\label{eq:Lameoperator}
\begin{split}
& A(x) \frac{d^2f}{dx^2}(x) 
+ B(x) \frac{df}{dx}(x)  - \frac14 \bigl( m(m+1)x + E\bigr) f(x) = 0, \\
& A(x) = (x-e_1)(x-e_2)(x-e_3), \\ & B(x) = 
\frac12\Bigl( (x-e_2)(x-e_3) + (x-e_1)(x-e_3) + (x-e_1)(x-e_2)
\Bigr).
\end{split}
\end{equation}
Note that the $e_i$'s are all different, where we follow the notation as 
in Whittaker and Watson \cite[\S 20, \S 20.32]{WhitW} for $\wp$, $e_i$, etc. 
In the form \eqref{eq:Lameoperator} it is a TD-operator. 
In \cite[\S 23.41]{WhitW} a related procedure is discussed which
leads to solutions of \eqref{eq:Lameoperator} for specific values
of $E$ by inserting descending power series in $x-e_2$.  
Another
classical line of study is to allow for a degree $p+1$ polynomial
in front of the second order derivative and a degree $p$ polynomial
in front of the first order derivative
in \eqref{eq:Lameoperator} and next to look for a polynomial, known
as the Van Vleck polynomial, of
degree $p-1$ in front of $f(x)$ in \eqref{eq:Lameoperator} such
that \eqref{eq:Lameoperator} has a polynomial solution $S(x)$,
known as the Heine-Stieltjes polynomial, see 
\cite[\S 6.8]{Szeg} and references for this line of
considerations. 

As noted \eqref{eq:Lameoperator} is a TD-operator,
which we now tridiagonalize. In light of Bochner's Theorem 
\ref{thm:Bochner}, the Al-Salam--Chihara Theorem \ref{thm:AlSalamChihara}
and the procedure sketched in \S \ref{ssec:thmBochnerAlSalamChihara}, 
we first use an affine transformation 
$x=ay+b$, $a=\frac12 (e_1-e_2)$, $b=\frac12(e_1+e_2)$, so
that $y=1$ corresponds $x=e_1$, $y=-1$ corresponds $x=e_2$. 
Note that we can use any other permutation of the points $e_1$, $e_2$ and
$e_3$. This yields 
\begin{equation*}
\begin{split}
&(y-1)(y+1)(y - \al)\frac{d^2g}{dy^2}(y)  
+ \frac12 \Bigl( (y+1)(y-\al) + (y-1)(y-\al) + (y-1)(y+1)\Bigr) \frac{dg}{dy}(y)  \\
&- \frac14 \bigl( m(m+1)(y+\frac{b}{a}) + E \bigr)g(y)=0
\end{split}
\end{equation*}
with $\al= -\frac{e_1+e_2-2e_3}{e_1-e_2}=\frac{3e_3}{e_1-e_2}\not=\pm 1$,
and $g(y)=f(\frac{x-b}{a})$. 
Let us denote the second order differential
operator for $E=0$ by $L$. 
In view of Bochner's Theorem 
\ref{thm:Bochner} and the factor $y^2-1$ in front of the second order
derivative, we try to tridiagonalize the operator using the Jacobi
polynomials $P_n^{(\al,\be)}$ and its second order
differential equation, see  \cite[(4.2.6)]{Isma}, \cite[(1.8.5)]{KoekS}. 
In this way we get rid of the
second order derivative, and collecting the remaining terms 
in front of the first order derivative gives 
\[
(y-\al)\bigl((\be-\al)-y(\al+\be+2)\bigr) 
+ \frac12 \Bigl( (y+1)(y-\al) + (y-1)(y-\al) + (y-1)(y+1)\Bigr),
\]
so that we can use the Al-Salam--Chihara Theorem \ref{thm:AlSalamChihara} 
in case this is a multiple of $(y^2-1)$. So this expression has
to be zero for $y=\pm 1$, and we find for the Jacobi polynomial 
parameters $\al=\be=-\frac12$, i.e. we
have to take the Chebychev polynomials $T_n$ in order to tridiagonalize
the Lam\'e equation. So using the Al-Salam--Chihara Theorem
\ref{thm:AlSalamChihara} and the three-term recurrence for the
Chebychev polynomials $T_n$, see e.g. \cite[\S 4.5]{Isma}, \cite[(1.8.34)]{KoekS}, 
we obtain
\begin{equation}\label{eq:LametridiagbyTn}
\begin{split}
L\, T_n = &\, \frac18 (2n-m)(2n+m+1) T_{n+1} 
+ \Bigl ( -\al n^2  -\frac14 m(m+1) \frac{b}{a}\Bigr) T_n 
\\ &\, +  \frac18(2n+m)(2n-m-1) T_{n-1} 
\end{split}
\end{equation}
for $n\geq 1$ and for  $n=0$ 
\begin{equation}\label{eq:LametridiagbyT0}
L\, T_0 = - \frac14 \bigl( m(m+1)(y+\frac{b}{a})\bigr)T_0 = 
- \frac14 \bigl( m(m+1)\bigr)T_1 - \frac14 \bigl( m(m+1)\frac{b}{a}\bigr)T_0. 
\end{equation}
Note that we cannot consider \eqref{eq:LametridiagbyT0} 
as the special case $n=0$ of \eqref{eq:LametridiagbyTn}. 
Note also that  \eqref{eq:LametridiagbyTn} and \eqref{eq:LametridiagbyT0}
exhibit the symmetry $m \leftrightarrow -m-1$. It is to be noted that
the recurrence \eqref{eq:LametridiagbyTn} can be solved using the
continuous dual $q$-Hahn polynomials, see \cite[\S 1.3]{KoekS}, precisely for the
excluded(!) values $\al=\pm 1$. Now for the
Lam\'e equation we need to solve $L\, \psi = E\, \psi$.

\begin{remark} It should be noted that this relation between the Lam\'e operator
and the Chebychev polynomials is conceptually different from 
a link discussed in Finkel et al. \cite{FinkGLR}, which
is related to the results of Ince \cite{Ince}. 
Finkel et al. \cite{FinkGLR} use the Jacobian version of the
Lam\'e operator, whereas we use the algebraic form, see
\cite[\S 23.4]{WhitW}. Their approach is motivated from the
theory of quasi-exactly solvable Hamiltonians, see 
\cite{GonzKO} for an overview.
\end{remark}

From \eqref{eq:LametridiagbyTn} it is clear that the
Jacobi matrix for the Lam\'e operator splits in case
$m\in \Z$. We will only discuss the case $m\in\N$ is even,
since we obtain a finite dimensional invariant
subspace. This is done in Section \ref{sssec:casemeven}.
In Section \ref{sssec:caseorthonormal} we consider 
a special case in which no coefficients in 
\eqref{eq:LametridiagbyTn}--\eqref{eq:LametridiagbyT0}
vanish and such that we can write 
$L$ in a symmetric form.

\subsubsection{Case $m=2k\in\N$ is even}\label{sssec:casemeven}
Let us first consider the case of $m=2k$ is even, then the 
Lam\'e operator $L$ leaves the $k+1$-dimensional space spanned by
$T_n$, $n=0,\ldots,k$, invariant. We can rewrite 
\eqref{eq:LametridiagbyTn} in this case as
\begin{equation}\label{eq:Lametridiagformeven}
\begin{split}
L T_n = &\Bigl( \frac18 (2n+1)(2n) -\frac18 (2k)(2k+1)\Bigr) T_{n+1}
+ \Bigl ( -\al  n^2  -\frac14 2k(2k+1) \frac{b}{a}\Bigr) T_n \\
&\qquad + \Bigl( \frac18 (2n)(2n-1) -\frac18 2k(2k+1)\Bigr) T_{n-1}, 
\qquad n\geq 1,  \\
= & - \frac14 \bigl( m(m+1)\bigr)T_1(y) - 
\frac14 \bigl( m(m+1)\frac{b}{a}\bigr)T_0(y), \qquad n=0.
\end{split}
\end{equation}
We now look for eigenfunctions $L \sum_{n=0}^k P_n(E) T_n = E \sum_{n=0}^k P_n(E) T_n$,
so we need
\begin{equation*}
\begin{split}
E\, P_0(E) = & \Bigl( \frac14 -\frac18 2k(2k+1)\Bigr) P_1(E)  
- \frac14 \bigl( m(m+1)\frac{b}{a}\bigr)P_0(E), \\
E\, P_n(E) = & \Bigl( \frac18 (2n+2)(2n+1) -\frac18 2k(2k+1)\Bigr) P_{n+1}(E)
+ \Bigl ( -\al n^2  -\frac14 2k(2k+1) \frac{b}{a}\Bigr) P_n(E) \\
&\qquad + 
\Bigl( \frac18 (2n-1)(2n-2) -\frac18 (2k)(2k+1)\Bigr) P_{n-1}(E),
\qquad 1\leq n\leq k, \\
E\, P_k(E) = &  \Bigl ( -\al k^2  -\frac14 2k(2k+1) \frac{b}{a}\Bigr) P_k(E)  
+ \Bigl( \frac18 (2k-1)(2k-2) -\frac18 (2k)(2k+1)\Bigr) P_{k-1}(E).
\end{split}
\end{equation*}
The possible values for the eigenvalue $E$ are determined as follows;
generate the polynomials $P_n$ by the first two equations
starting with $P_0(E)=1$. Stop at $P_{k+1}(E)$, and then
the zeroes of $P_{k+1}(E)$ are the only possible eigenvalues
of the Lam\'e operator restricted to this finite-dimensional
space.  This corresponds nicely to \cite[\S 23.41]{WhitW}.
Note moreover that for $\al\in\R$ the spectral theorem
for orthogonal polynomials \cite{Isma} (also known as
Favard's theorem) is valid,
implying that the polynomials $P_n$, $n=0,\ldots, k$, are
orthogonal with respect to a measure on the real line, so
that there are $k+1$ different
real eigenvalues $E$. 

\subsubsection{Orthonormal version}\label{sssec:caseorthonormal}
Assuming that there are only non-zero coefficients in the three-term
relation \eqref{eq:LametridiagbyTn}--\eqref{eq:LametridiagbyT0} 
we can ask under what conditions
there exists an orthonormal version. 
Assume $m\in (2k+1,2k+2)$ for some $k\in \N$, and 
put $T_n=\al_n p_n$ with
\[
\al_n = \sqrt{\frac{(\frac12(1-m))_n (1+ \frac12 m)_n}{(-\frac12m)_n(\frac12(m+1))_n}}.
\]
Note that the condition on $m$ implies that the argument of the
square root is indeed positive. Then \eqref{eq:LametridiagbyTn}--\eqref{eq:LametridiagbyT0} 
is rewritten as 
\begin{equation}\label{eq:Lameorthonormal}
\begin{split}
L\, p_n = &\, a_n\, p_{n+1} + b_n\, p_n + a_{n-1} \, p_{n-1}\\ 
\\ a_n = &\, \frac12\sqrt{(n+\frac12m+1)(n-\frac12 m+\frac12)(n-\frac12 m) (n+\frac12 m+\frac12)}, 
\\ b_n  =&\,  -\al n^2  -\frac14 m(m+1) \frac{b}{a}, \\
L\, p_0 = &\, 2\,a_0\, p_1 + b_0\, p_0, 
\end{split}
\end{equation}
which can be viewed as a symmetric operator assuming $\al\in\R$ except for the 
the last line in \eqref{eq:Lameorthonormal}. 
At this point it is not clear if the Jacobi form of $L$
as displayed by the first equality in \eqref{eq:Lameorthonormal}
gives rise to an essentially self-adjoint operator or not, since
the coefficients $a_n = {\mathcal O}(n^2)$, $b_n = {\mathcal O}(n^2)$ 
blow up.
Since 
\begin{equation*}
a_n\, = \, \frac12 n^2\bigl( 1 + \frac{1}{n} + \frac14 (\frac12 -m(m+1))\frac{1}{n^2} + \mathcal{O}(\frac{1}{n^4})\bigr),
\qquad n\to\infty
\end{equation*}
we find 
\begin{equation}\label{eq:essselfadjointLame}
a_n\, + \, a_{n-1}\, \pm \, b_n \, = \, (1\mp \al)n^2 + \frac18 - (1\pm \frac{b}{a})m(m+1) + \mathcal{O}(\frac{1}{n^2}),
\end{equation}
so that for any $\al\in\R$ one of the expressions in \eqref{eq:essselfadjointLame} is bounded
from above, so that \cite[Ch.~VII, Thm.~1.4, Cor.]{Bere} implies that $L$ is essentially
self-adjoint in all cases.

Note that in this case
\eqref{eq:firstorderodeweight} gives
\[
(\ln w)' = \frac{B-A'}{A} = -\frac12 \left( \frac{1}{y-1} +
\frac{1}{y+1}  + \frac{1}{y-\al}\right)
\]
so that we take $w$ to be a multiple of 
\[
\frac{1}{\sqrt{(y^2-1)}} \frac{1}{\sqrt{y-\al}}.
\]
Since the coefficients $a_n$ and $b_n$ are unbounded, the orthogonality measure
has unbounded support. In such a case we see that $w$ does not satisfy the 
conditions of Lemma \ref{lem:diffeqLsymmetricpols}, since the moments do not exist.
A further study of the orthonormal case seems to be required. 

\begin{remark}
One is tempted to speculate about the case of the original differential operator  $L$ not being essentially selfadjoint.  We have associated with a tridiagonalizable $L$ a family of orthogonal polynomials. When $L$ is not essentially selfadjoint the corresponding orthogonal polynomials are orthogonal with respect to infinitely many measures, that is  the 
corresponding moment problem   will be indeterminate \cite{Akh}. The orthogonal polynomials will be orthogonal to measures which are not spectral measures of the Jacobi form of $L$.   
\end{remark}

\emph{Acknowledgement.} We thank the referee for useful remarks and comments. 
The research of Mourad E.H. Ismail is supported by a grant from King Saud
University, Saudi Arabia.


\end{document}